\documentclass[11pt]{article}

\usepackage[margin=1in]{geometry}
\usepackage{amsmath,amssymb,amsthm,mathtools}
\usepackage{microtype}
\usepackage{enumitem}
\usepackage[colorlinks=true,linkcolor=blue,citecolor=blue,urlcolor=blue]{hyperref}

\newtheorem{theorem}{Theorem}
\newtheorem{proposition}[theorem]{Proposition}
\newtheorem{lemma}[theorem]{Lemma}
\newtheorem{corollary}[theorem]{Corollary}
\theoremstyle{definition}

\theoremstyle{remark}
\newtheorem{remark}[theorem]{Remark}

\newcommand{\R}{\mathbb{R}}
\newcommand{\Tau}{\mathcal{T}}

\DeclareMathOperator{\cone}{cone}
\DeclareMathOperator{\conv}{conv}
\DeclareMathOperator{\diam}{diam}
\DeclareMathOperator{\dist}{dist}
\DeclareMathOperator{\rad}{rad}
\DeclareMathOperator{\children}{Ch}

\begin{document}

\title{An improved bound for sumsets of thick compact sets\\via the Shapley--Folkman theorem}
\author{Scott Duke Kominers\thanks{\textsl{Harvard Business School; Department of Economics and Center of Mathematical Sciences and Applications, Harvard University; and a16z crypto.}}}
\date{}

\maketitle

\begin{abstract}
Let \(E_1,\dots,E_n \subset \R^d\) be compact sets of positive diameter with Feng--Wu thickness at least \(c>0\).
Feng and Wu proved that \(E_1+\cdots+E_n\) has non-empty interior when \(n>2^{11}c^{-3}+1\).
We show that
\[
n>\frac{\sqrt d}{(\sqrt{1+c}-1)^2}
=\frac{\sqrt d\,(\sqrt{1+c}+1)^2}{c^2}
\]
already suffices. 
In particular, since \(0<c\le 1\), the bound \(n>6\sqrt d\,c^{-2}\) is enough.
For fixed dimension~\(d\), this improves the exponent in \(c^{-1}\) from~\(3\) to~\(2\), while introducing only an explicit factor of~\(\sqrt d\).
The proof replaces the one-summand-at-a-time enlargement of Feng--Wu by a simultaneous convexification step
based on a radius form of the Shapley--Folkman theorem.
\end{abstract}

\section{Introduction}

A classical theme in additive combinatorics and convex geometry is to understand when the Minkowski sum of ``large'' and/or structured sets must itself be large. In the discrete setting, inverse results such as the Freiman--Ruzsa theorem and its descendants~\cite{Freiman,Ruzsa} give structural control of sums of finite sets with small doubling. In the continuous setting, the Steinhaus theorem~\cite{Steinhaus} guarantees that the sum of two sets of positive Lebesgue measure contains a ball. Between these lies an intermediate regime---compact sets of Lebesgue measure zero with controlled multiscale geometric structure---which is considerably more delicate and connects to several active areas, including sumsets and projections of self-similar and self-conformal sets~\cite{Shmerkin}, Diophantine approximation problems with Cantor-type targets~\cite{KLW}, and homoclinic bifurcations for surface diffeomorphisms, where arithmetic sums of dynamically defined Cantor sets govern transitions in the number of coexisting attractors (the Palis conjecture~\cite{Palis}, with major progress due to Moreira and Yoccoz~\cite{MY}).

For non-empty
\(E_1,\dots,E_n\subset \R^d\), we write
\[
E_1+\cdots+E_n
=
\{x_1+\cdots+x_n : x_i\in E_i\}
\]
for their Minkowski sum. Throughout, \(B(x,r):=\{y\in \R^d: |y-x|\le r\}\) denotes the closed Euclidean ball of
radius \(r\) centered at \(x\).

Following Feng and Wu~\cite{FW}, the \emph{thickness} of a compact set
\(E\subset \R^d\) is defined by
\[
\tau(E)
=
\sup\Bigl\{c\in[0,1]:
\forall x\in E,\ \forall\,0<r\le \diam(E),\
\exists y\in\R^d\ \text{with}\
B(y,cr)\subset \conv(E\cap B(x,r))
\Bigr\}.
\]
Since the admissible constants in the definition are downward closed, the condition
\(\tau(E)\ge c\) implies that for every \(c'\) with \(0<c'<c\), every \(x\in E\), and every
\(0<r\le \diam(E)\), there exists \(y\in \R^d\) such that
\[
B(y,c'r)\subset \conv(E\cap B(x,r)).
\]
In words, at every point and every scale, the set~\(E\) contains a portion whose convex
hull includes a ball of radius proportional to the ambient scale, with any
proportionality constant strictly below~\(c\). Feng and Wu~\cite[Theorem 1.2]{FW} proved
that if \(\diam(E_i)>0\) and \(\tau(E_i)\ge c>0\) for each~\(i\), then \(E_1+\cdots+E_n\)
has non-empty interior provided that \(n>2^{11}c^{-3}+1\). Their argument
proceeds by inductively replacing one summand at a time: at each generation of a
tree-like multiscale decomposition, a single summand's finite point cloud is replaced by
a ball via a geometric covering estimate (their Lemma~4.1 and Corollary~3.3), and the
resulting loss accumulates over generations.

In this note we give a proof with a different structure that yields a
stronger quantitative bound for small \(c\).

\begin{theorem}\label{thm:main}
Let \(E_1,\dots,E_n\subset \R^d\) be compact sets with \(\diam(E_i)>0\) and
\[
\tau(E_i)\ge c>0
\qquad (1\le i\le n).
\]
If we have
\[
n>\frac{\sqrt d}{(\sqrt{1+c}-1)^2}
=
\frac{\sqrt d\,(\sqrt{1+c}+1)^2}{c^2},
\]
then \(E_1+\cdots+E_n\) has non-empty interior.

In particular, since \(0<c\le 1\),
\[
n>6\sqrt d\,c^{-2}
\]
suffices.
\end{theorem}

\begin{remark}[Positive-diameter hypothesis]\label{rem:singletons}
The condition \(\diam(E_i)>0\) in Theorem~\ref{thm:main} (and later in
Proposition~\ref{prop:parametric-thick-sum}) is a genuine hypothesis rather than a harmless
normalization. A singleton summand \(E_j=\{a_j\}\) merely translates the
sum---\(E_1+\cdots+E_n=(\sum_{i\neq j}E_i)+a_j\)---but removing it also reduces
the total count of summands from~\(n\) to~\(n-1\). In a counting-based result this
matters: the remaining \(n-1\) summands may fail the threshold even if the
original~\(n\) summands satisfied it. Note moreover that singletons satisfy
\(\tau(E)=1\) vacuously under the Feng--Wu definition, as the quantifier
``\(\forall\,0<r\le\diam(E)\)'' ranges over an empty set when \(\diam(E)=0\); thus
singletons carry maximal formal thickness yet contribute no geometric content to
the sum.

Accordingly, we state the theorem only for families of positive-diameter sets;
if some singleton summands are present, the conclusions below apply provided
the subfamily of non-singleton summands already satisfies the stated threshold.
\end{remark}

\begin{remark}[Comparison with Feng--Wu]\label{rem:comparison}
For families of compact sets of positive diameter and common thickness at least~\(c\), the
currently available sufficient thresholds are the Feng--Wu bound
\(n>2^{11}c^{-3}+1\) and the quadratic bound of Theorem~\ref{thm:main}. Hence, under
the common nondegeneracy hypothesis,
\[
n>\min\!\left\{2^{11}c^{-3}+1,\
\frac{\sqrt d\,(\sqrt{1+c}+1)^2}{c^2}\right\}
\]
is a sufficient condition for \(E_1+\cdots+E_n\) to have non-empty interior.

A short computation shows that \(6\sqrt d\,c^{-2}<2^{11}c^{-3}\) whenever
\[
d<\left(\tfrac{1024}{3}\right)^2 c^{-2}\approx 1.17\times 10^5\cdot c^{-2}.
\]
Thus, when \(d\) is at most of order \(c^{-2}\), our new quadratic threshold is already
the smaller of the two; in particular, for each fixed~\(d\), our bound dominates for all
sufficiently small~\(c\).
\end{remark}

The key new ingredient in our analysis is a radius form of the Shapley--Folkman theorem
(see~\cite{AH,Starr}), a classical result in mathematical economics and convex geometry,
which asserts that the Minkowski sum of many sets is approximately convex, with an
approximation error depending on the ambient dimension and the radii of the individual
sets, but not on the number of summands. In the Feng--Wu argument, convexification is
performed sequentially, one summand at a time; each step introduces a geometric slack
(controlled by \(c\)) that is re-incurred through the scale recursion. Since there are about
\(c^{-1}\) such steps per generation, this slack accumulates through the multiscale iteration,
leading to their \(c^{-3}\) threshold. By applying Shapley--Folkman to convexify all~\(n\)
summands at once, we replace this sequential accumulation by a single additive error term
of size~\(O(\sqrt d\,r_k)\) at scale \(r_k\), which is exactly what improves the exponent
from \(3\) to \(2\).
Our proof also incorporates two further simplifications over~\cite{FW}:

\begin{enumerate}[label=\textup{(\roman*)}]
\item We never need an explicit packing or covering number bound such as Feng--Wu's
Lemma~4.1; compactness alone provides a finite discretization of each thick piece at the
required scale.
\item The support-function perturbation lemma (Lemma~\ref{lem:support-perturbation} in the sequel) gives a
clean and transparent way to pass from a finite approximation to a convex hull containing
a ball, without the volumetric estimates used in~\cite{FW}.
\end{enumerate}

\section{Key ingredients}

\subsection{A support-function perturbation lemma}

Our first lemma captures a simple but useful stability principle: if a
compact set~\(A\) has a convex hull containing a ball, and we approximate~\(A\) by a
coarser finite set~\(F\)---in the sense that every point of~\(A\) lies within
\(\varepsilon\) of some point of~\(F\)---then the convex hull of~\(F\) still contains
a ball, merely shrunk by~\(\varepsilon\) in radius.

For a non-empty compact set \(K\subset \R^d\), let
\[
h_K(u):=\sup_{x\in K}\langle x,u\rangle
\qquad (u\in \R^d)
\]
denote its support function.

\begin{lemma}\label{lem:support-perturbation}
Let \(A,F\subset \R^d\) be non-empty compact sets. Assume that
\[
A\subset F+B(0,\varepsilon)
\]
for some \(\varepsilon>0\), and that
\[
B(z,r)\subset \conv(A)
\]
for some \(z\in \R^d\) and \(r>\varepsilon\). Then
\[
B(z,r-\varepsilon)\subset \conv(F).
\]
\end{lemma}

\begin{proof}
Since \(A\subset F+B(0,\varepsilon)\), for every unit vector \(u\in \R^d\) we have
\[
h_A(u)\le h_F(u)+\varepsilon.
\]
As support functions are unchanged by convexification,
\begin{equation}
h_{\conv(A)}(u)\le h_{\conv(F)}(u)+\varepsilon.
\label{eq:support-upper}
\end{equation}
On the other hand, the hypothesis
\[
B(z,r)\subset \conv(A)
\]
implies that
\begin{equation}
h_{\conv(A)}(u)\ge h_{B(z,r)}(u)=\langle z,u\rangle+r
\label{eq:support-lower}
\end{equation}
for every unit vector~\(u\). Combining inequalities \eqref{eq:support-upper} and \eqref{eq:support-lower}, we obtain
\[
h_{\conv(F)}(u)\ge \langle z,u\rangle+r-\varepsilon
=h_{B(z,r-\varepsilon)}(u).
\]
Therefore \(B(z,r-\varepsilon)\subset \conv(F)\), as claimed.
\end{proof}

\subsection{Finite discretization of thickness}

The following lemma is in some sense the ``workhorse'' of our argument: it converts the thickness condition---which a priori involves the convex hull of the (possibly infinite) set \(E\cap B(x,r)\)---into a statement about finitely many points. This finite discretization is essential: the tree construction in the sequel operates on finite point configurations, so we need to know that at every location and scale, a \emph{finite} subset of~\(E\) already has a convex hull containing a nearly full-sized ball. Compactness supplies the finite approximation, and
Lemma~\ref{lem:support-perturbation} controls the radius lost in passing from the full intersection to the finite sample.

This is the only place where the thickness hypothesis enters the argument directly.

\begin{lemma}\label{lem:finite-discretization}
Let \(E\subset \R^d\) be compact with \(\diam(E)>0\) and \(\tau(E)\ge c>0\). Fix any
\(\alpha\) with
\[
0<\alpha<c.
\]
Then for every \(x\in E\) and every \(0<r\le \diam(E)\), there exist a finite set
\[
Y\subset E\cap B(x,r)
\]
and a point \(z\in \R^d\) such that
\[
B(z,\alpha r)\subset \conv(Y).
\]
\end{lemma}

\begin{proof}
Set \(A:=E\cap B(x,r)\), which is compact because \(E\) is compact and \(B(x,r)\) is
closed. Choose any \(\beta\) with
\[
\alpha<\beta<c.
\]
Since \(\tau(E)\ge c\), the discussion following the definition of thickness implies that
there exists \(z\in \R^d\) such that
\[
B(z,\beta r)\subset \conv(A).
\]
Because \(A\) is compact, there exists a finite set \(Y\subset A\) such that
\[
A\subset Y+B\bigl(0,(\beta-\alpha)r\bigr).
\]
Applying Lemma~\ref{lem:support-perturbation} with \(\varepsilon=(\beta-\alpha)r\), we obtain
\[
B(z,\alpha r)\subset \conv(Y). \qedhere
\]
\end{proof}

\subsection{The Shapley--Folkman theorem in radius form}

The Shapley--Folkman theorem originated in the study of equilibria in economies with
non-convex preferences~(see \cite{AH,Starr}). For completeness, we give a self-contained proof of the radius form
needed here. The argument passes through the classical exceptional-summand form of
Shapley--Folkman and then uses a simple probabilistic rounding lemma.

\begin{lemma}[Conic Carath\'eodory]\label{lem:conic-caratheodory}
Let \(S\subset \R^m\), and let \(v\in \cone(S)\). Then \(v\) can be written
as a non-negative linear combination of at most \(m\) vectors of~\(S\).
\end{lemma}

\begin{proof}
If \(v=0\), there is nothing to prove. So we assume \(v\neq 0\) and choose a representation
\[
v=\sum_{j=1}^N \lambda_j s_j,
\qquad
\lambda_j>0,\ s_j\in S,
\]
with \(N\) minimal. If \(N>m\), then \(s_1,\dots,s_N\) are linearly dependent, so there
exist real numbers \(a_1,\dots,a_N\), not all zero, with \(\sum_{j=1}^N a_j s_j=0\). At
least one \(a_j\) is positive. Let
\[
t:=\min_{a_j>0}\frac{\lambda_j}{a_j}>0.
\]
Then \(\lambda_j':=\lambda_j-ta_j\ge 0\) for all~\(j\), and \(\lambda_j'=0\) for at least
one index. Moreover, \(v=\sum_{j=1}^N \lambda_j' s_j\). Discarding zero coefficients
yields a representation with fewer than \(N\) terms, contradicting minimality. Hence
\(N\le m\).
\end{proof}

We first recall the classical form of Shapley--Folkman.

\begin{lemma}[Shapley--Folkman, classical form]\label{lem:shapley-folkman-classical}
Let \(A_1,\dots,A_n\subset \R^d\) be non-empty compact sets. For every
\[
x\in \conv(A_1)+\cdots+\conv(A_n)
\]
there exists a set \(I\subset \{1,\dots,n\}\) with \(|I|\le d\) such that
\[
x\in \sum_{i\notin I} A_i+\sum_{i\in I}\conv(A_i).
\]
\end{lemma}

\begin{proof}
Write \(x=x_1+\cdots+x_n\) with \(x_i\in \conv(A_i)\). Let \(e_1,\dots,e_n\) be the
standard basis of~\(\R^n\), and define
\[
S:=\bigcup_{i=1}^n \bigl(A_i\times \{e_i\}\bigr)\subset \R^{d+n}.
\]
Then \((x,e_1+\cdots+e_n)\in \cone(S)\). By Lemma~\ref{lem:conic-caratheodory}, there
exist indices \(i_1,\dots,i_m\in \{1,\dots,n\}\), points
\(a_\ell\in A_{i_\ell}\) for \(1\le \ell\le m\), and coefficients \(\lambda_\ell\ge 0\),
with \(m\le d+n\), such that
\[
(x,e_1+\cdots+e_n)=\sum_{\ell=1}^m \lambda_\ell (a_\ell,e_{i_\ell}).
\]
For each \(i\in \{1,\dots,n\}\), the \(i\)-th auxiliary coordinate gives
\[
\sum_{\ell:\, i_\ell=i}\lambda_\ell=1.
\]
Therefore each index~\(i\) appears at least once among \(i_1,\dots,i_m\). Since
\(m\le d+n\), at most \(d\) indices can appear more than once. Let \(I\) be the set of
indices appearing more than once, so \(|I|\le d\).

For each \(i\notin I\), there is exactly one \(\ell\) with \(i_\ell=i\), and its coefficient is
\(1\); denote the corresponding point of \(A_i\) by \(a_i\). For each \(i\in I\), set
\[
y_i:=\sum_{\ell:\, i_\ell=i}\lambda_\ell a_\ell.
\]
Then \(y_i\in \conv(A_i)\). Collecting terms gives
\[
x=\sum_{i\notin I} a_i+\sum_{i\in I} y_i
\in \sum_{i\notin I} A_i+\sum_{i\in I}\conv(A_i). \qedhere
\]
\end{proof}

The next lemma is the Euclidean-radius ingredient that sharpens the usual diameter-based
estimate.

\begin{lemma}[Rounding a convexified sum in radius form]\label{lem:radius-rounding}
Let \(A_1,\dots,A_m\subset \R^d\) be non-empty compact sets, and suppose that for some
points \(c_1,\dots,c_m\in \R^d\) and some \(R\ge 0\),
\[
A_i\subset B(c_i,R)
\qquad (1\le i\le m).
\]
Then for every choice of points
\[
y_i\in \conv(A_i)
\qquad (1\le i\le m),
\]
there exist points \(a_i\in A_i\) such that
\[
\left|\sum_{i=1}^m y_i-\sum_{i=1}^m a_i\right|
\le R\sqrt m.
\]
\end{lemma}

\begin{proof}
Replacing each \(A_i\) by \(A_i-c_i\), each \(y_i\) by \(y_i-c_i\), and later translating
back, we may assume that \(c_i=0\) for all~\(i\). Thus
\[
A_i\subset B(0,R)
\qquad (1\le i\le m).
\]

Fix \(i\). Since \(y_i\in \conv(A_i)\), we have \((y_i,1)\in \cone(A_i\times\{1\})\subset\R^{d+1}\).
By Lemma~\ref{lem:conic-caratheodory}, there exist \(a_{ij}\in A_i\) and \(\lambda_{ij}\ge 0\) such that
\[
(y_i,1)=\sum_{j=1}^{N_i}\lambda_{ij}(a_{ij},1).
\]
Comparing the last coordinates gives \(\sum_{j=1}^{N_i}\lambda_{ij}=1\), and comparing the first \(d\)
coordinates gives \(y_i=\sum_{j=1}^{N_i}\lambda_{ij}a_{ij}\).

Let \(X_i\) be an \(\R^d\)-valued random variable taking the value \(a_{ij}\) with
probability \(\lambda_{ij}\), and assume that \(X_1,\dots,X_m\) are independent. Then
\[
\mathbb{E}[X_i]=y_i.
\]
Also, we have
\[
\mathbb{E}|X_i-y_i|^2
=
\mathbb{E}|X_i|^2-|y_i|^2
\le
\mathbb{E}|X_i|^2
\le R^2.
\]
Since the random vectors \(X_i-y_i\) are independent and have mean zero,
\[
\mathbb{E}\left|\sum_{i=1}^m (X_i-y_i)\right|^2
=
\sum_{i=1}^m \mathbb{E}|X_i-y_i|^2
\le mR^2.
\]
Therefore there exists a realization \(a_i\in A_i\) of the \(X_i\) such that
\[
\left|\sum_{i=1}^m y_i-\sum_{i=1}^m a_i\right|
=
\left|\sum_{i=1}^m (y_i-a_i)\right|
\le R\sqrt m. \qedhere
\]
\end{proof}

We can now state and prove the radius form of Shapley--Folkman used below.

For a non-empty bounded set \(A\subset \R^d\), we define the radius by
\[
\rad(A):=\inf_{x\in\R^d}\sup_{a\in A}|a-x|;
\]
for compact \(A\), the infimum is attained because the function
\(x\mapsto \sup_{a\in A}|a-x|\) is continuous and tends to infinity as \(|x|\to\infty\).

\begin{theorem}[Shapley--Folkman, radius form]\label{thm:shapley-folkman-radius}
Let \(A_1,\dots,A_n\subset \R^d\) be non-empty compact sets, and set
\[
R:=\max_{1\le i\le n}\rad(A_i).
\]
Then for every
\[
x\in \conv(A_1)+\cdots+\conv(A_n)
=
\conv(A_1+\cdots+A_n),
\]
there exist points \(a_i\in A_i\) such that
\[
\left|x-(a_1+\cdots+a_n)\right|
\le
R\sqrt{\min\{n,d\}}.
\]
\end{theorem}

\begin{proof}
By Lemma~\ref{lem:shapley-folkman-classical}, there exists a set \(I\subset \{1,\dots,n\}\) with
\(|I|\le d\) such that
\[
x\in \sum_{i\notin I} A_i+\sum_{i\in I}\conv(A_i).
\]
Thus, we may write
\[
x=\sum_{i\notin I} a_i+\sum_{i\in I} y_i
\]
with \(a_i\in A_i\) for \(i\notin I\) and \(y_i\in \conv(A_i)\) for \(i\in I\).

If \(I=\varnothing\), then \(x\in A_1+\cdots+A_n\) and there is nothing to prove. So
assume \(I\neq\varnothing\), and let \(q:=|I|\). Since each \(A_i\) is compact, there
exists \(c_i\in \R^d\) such that
\[
A_i\subset B(c_i,\rad(A_i))\subset B(c_i,R)
\qquad (i\in I).
\]
Applying Lemma~\ref{lem:radius-rounding} to the family \((A_i)_{i\in I}\), we obtain points
\(a_i\in A_i\) for \(i\in I\) such that
\[
\left|\sum_{i\in I} y_i-\sum_{i\in I} a_i\right|
\le R\sqrt q.
\]
Hence, we have
\[
\left|x-\sum_{i=1}^n a_i\right|
=
\left|\sum_{i\in I} y_i-\sum_{i\in I} a_i\right|
\le R\sqrt q
\le R\sqrt{\min\{n,d\}},
\]
as claimed.
\end{proof}

The preceding theorem immediately yields a quantitative ``almost convexity''
estimate for Minkowski sums.

\begin{corollary}\label{cor:almost-convex}
Let \(A_1,\dots,A_n\subset \R^d\) be non-empty compact sets, and suppose that for some
points \(c_1,\dots,c_n\in \R^d\) and some \(R\ge 0\),
\[
A_i\subset B(c_i,R)
\qquad (1\le i\le n).
\]
Then, we have
\[
\conv(A_1)+\cdots+\conv(A_n)
=
\conv(A_1+\cdots+A_n)
\subset
A_1+\cdots+A_n+B\bigl(0,R\sqrt{\min\{n,d\}}\bigr);
\]
in particular,
\[
\conv(A_1)+\cdots+\conv(A_n)
\subset
A_1+\cdots+A_n+B(0,R\sqrt d).
\]
\end{corollary}

\begin{proof}
Since \(A_i\subset B(c_i,R)\), we have \(\rad(A_i)\le R\) for each~\(i\). The claim is
therefore immediate from Theorem~\ref{thm:shapley-folkman-radius}.
\end{proof}

\begin{remark}\label{rem:almost-convex-coarse}
The conclusion of Corollary~\ref{cor:almost-convex} is stated in a deliberately simple common-radius form.
The theorem actually yields the sharper residual \(R\sqrt{\min\{n,d\}}\), and one could
propagate that sharper quantity through the later tree argument. We use the coarser
bound \(R\sqrt d\) because it keeps the later inequalities transparent and already yields
the desired quadratic dependence on~\(c^{-1}\) together with a square-root dependence
on~\(d\).

More generally, the estimate \(R\sqrt d\) is not intended as a sharp approximation
constant for arbitrary compact sets; more refined approximate-convexity bounds are
available (see, e.g.,~\cite{Starr,Budish,WuTang}). We use the radius-based form presented here because
it is transparent and sufficient for our purposes.
\end{remark}

\subsection{Absorption of residuals}

Our next lemma is the key mechanism that connects the radius-form approximate convexity
estimate to the tree construction in the sequel. The setup is as follows: we have
\(n\) finite point clouds \(A_1,\dots,A_n\), each of whose convex hulls contains a ball
of radius~\(q\), and each of which lies in some ball of radius~\(R\). We would like to
conclude that the Minkowski sum of those balls sits inside the Minkowski sum of the point
clouds themselves---but of course it need not, since the point clouds are not convex.
Corollary~\ref{cor:almost-convex} tells us that the gap between the sum of the convex
hulls and the sum of the original sets is at most \(R\sqrt d\), an additive error
depending on the dimension and the common radius bound but not on~\(n\). This error is a
fixed cost, and there are \(n\) summands available to absorb it: if we fatten each point
cloud by a small radius~\(Q\), then the \(n\) copies of \(B(0,Q)\) collectively produce a
ball of radius~\(nQ\). Provided that \(nQ\ge R\sqrt d\), the collective fattening
swallows the Shapley--Folkman residual. Each summand bears only a \(\frac{1}{n}\)-share of the
dimensional cost, which is what ultimately allows the threshold on~\(n\) to scale as
\(\frac{\sqrt d}{c^2}\) rather than \(\frac{1}{c^3}\).

\begin{lemma}[One-step absorption]\label{lem:one-step-absorption}
Let \(A_1,\dots,A_n\subset \R^d\) be non-empty finite sets, and let
\[
b_1,\dots,b_n,c_1,\dots,c_n\in \R^d.
\]
Suppose that for some \(q,R>0\),
\[
B(b_i,q)\subset \conv(A_i)
\quad\text{and}\quad
A_i\subset B(c_i,R)
\qquad (1\le i\le n),
\]
and moreover assume that
\[
R\sqrt d\le nQ
\]
for some \(Q>0\). Then, we have
\[
B(b_1,q)+\cdots+B(b_n,q)
\subset
\left(\bigcup_{x_1\in A_1} B(x_1,Q)\right)+\cdots+
\left(\bigcup_{x_n\in A_n} B(x_n,Q)\right).
\]
\end{lemma}

\begin{proof}
From the hypothesis that \(B(b_i,q)\subset \conv(A_i)\) for each \(i\), we obtain
\[
B(b_1,q)+\cdots+B(b_n,q)
=
B(b_1+\cdots+b_n,nq)
\subset
\conv(A_1)+\cdots+\conv(A_n).
\]
By Corollary~\ref{cor:almost-convex},
\[
\conv(A_1)+\cdots+\conv(A_n)
\subset
A_1+\cdots+A_n+B(0,R\sqrt d).
\]
Since \(R\sqrt d\le nQ\),
\[
A_1+\cdots+A_n+B(0,R\sqrt d)
\subset
A_1+\cdots+A_n+B(0,nQ).
\]
Finally, we note that
\[
B(0,nQ)=\underbrace{B(0,Q)+\cdots+B(0,Q)}_{n\text{ times}},
\]
so we have
\begin{align*}
A_1+\cdots+A_n+B(0,nQ)
&=
(A_1+B(0,Q))+\cdots+(A_n+B(0,Q))\\
&=
\left(\bigcup_{x_1\in A_1} B(x_1,Q)\right)+\cdots+
\left(\bigcup_{x_n\in A_n} B(x_n,Q)\right).
\end{align*}
Combining the preceding series of inclusions proves the claim.
\end{proof}

\section{A parameterized quantitative thick-sum theorem}

We now prove the main inductive statement. The genuinely new ingredient has already been
isolated in Lemma~\ref{lem:one-step-absorption}; the recursive tree construction below serves to
produce, at each generation, finite point clouds to which that one-step absorption lemma
can be applied.

\begin{proposition}\label{prop:parametric-thick-sum}
Let \(E_1,\dots,E_n\subset \R^d\) be compact with \(\diam(E_i)>0\), and assume
\[
\tau(E_i)\ge c>0
\qquad (1\le i\le n);
\]
fix parameters \(\alpha,\lambda\) with
\[
0<\lambda<\alpha<c.
\]
If we have
\begin{equation}
n>\frac{\sqrt d\,(1+\lambda)}{\lambda(\alpha-\lambda)},
\label{eq:parametric-bound}
\end{equation}
then \(E_1+\cdots+E_n\) has non-empty interior.
\end{proposition}

\begin{proof}
Set
\[
r_0:=\min_{1\le i\le n}\diam(E_i)>0,
\qquad
r_k:=\lambda^k r_0
\quad (k\ge 0).
\]
For each \(i\in\{1,\dots,n\}\) we recursively construct a rooted tree
\[
\Tau^i=\bigsqcup_{k\ge 0}\Tau_k^i,
\qquad
\Tau_0^i=\{\varnothing\},
\]
where \(\Tau_k^i\) is the set of vertices at depth~\(k\) (the levels being pairwise
disjoint). To each vertex \(v\in \Tau_k^i\) we associate a point \(x_v^i\in E_i\) and a
point \(z_v^i\in \R^d\), as follows.

Choose arbitrary root points \(x_\varnothing^i\in E_i\). Suppose that \(x_v^i\) has been
defined for a vertex \(v\in \Tau_k^i\). Applying Lemma~\ref{lem:finite-discretization} to~\(E_i\), the
point~\(x_v^i\), and the scale~\(r_k\), we obtain a finite set
\[
\{x_u^i : u\in \children(v)\}\subset E_i\cap B(x_v^i,r_k),
\]
where \(\children(v)\subset \Tau_{k+1}^i\) denotes the set of children of~\(v\), together with a
point \(z_v^i\in \R^d\) such that
\begin{equation}\label{eq:alpha-ball}
B(z_v^i,\alpha r_k)\subset \conv\bigl(\{x_u^i:u\in \children(v)\}\bigr).
\end{equation}
This completes the recursive construction.

Set \(q_k:=(\alpha-\lambda)r_k\) for \(k\ge 0\).

\medskip
\noindent\textbf{Step 1: the children's \(z\)-points convexly span a ball of radius \(q_k\).}
Fix \(i\), \(k\), and \(v\in \Tau_k^i\). For each child \(u\in \children(v)\), the recursive
construction at the vertex \(u\) gives
\[
B(z_u^i,\alpha r_{k+1})
\subset
\conv\bigl(\{x_w^i:w\in \children(u)\}\bigr),
\]
while we have
\[
\{x_w^i:w\in \children(u)\}\subset B(x_u^i,r_{k+1}),
\]
and so
\[
B(z_u^i,\alpha r_{k+1})
\subset
\conv\bigl(\{x_w^i:w\in \children(u)\}\bigr)
\subset
B(x_u^i,r_{k+1}).
\]
In particular, \(B(z_u^i,\alpha r_{k+1})\subset B(x_u^i,r_{k+1})\) implies
\[
|z_u^i-x_u^i|+\alpha r_{k+1}\le r_{k+1};
\]
hence, we have
\[
|z_u^i-x_u^i|\le (1-\alpha)r_{k+1}\le r_{k+1}.
\]
Therefore, for each \(u\in \children(v)\) we have \(x_u^i\in z_u^i+B(0,r_{k+1})\), and thus
\[
\{x_u^i:u\in \children(v)\}\subset \{z_u^i:u\in \children(v)\}+B(0,r_{k+1}).
\]
Applying Lemma~\ref{lem:support-perturbation} to~\eqref{eq:alpha-ball} with
\[
A=\{x_u^i:u\in \children(v)\},\qquad
F=\{z_u^i:u\in \children(v)\},\qquad
\varepsilon=r_{k+1}=\lambda r_k,
\]
(which is admissible since \(\lambda<\alpha\)), we obtain
\begin{equation}\label{eq:q-ball}
B(z_v^i,\alpha r_k-r_{k+1})
=
B(z_v^i,q_k)
\subset
\conv\bigl(\{z_u^i:u\in \children(v)\}\bigr).
\end{equation}

\medskip
\noindent\textbf{Step 2: radius control.}
Fix \(i\), \(k\), and \(v\in \Tau_k^i\), and let \(u\in \children(v)\). As above,
\[
B(z_u^i,\alpha r_{k+1})\subset B(x_u^i,r_{k+1}),
\]
and hence
\[
|z_u^i-x_u^i|\le (1-\alpha)r_{k+1}.
\]
Therefore, we have
\[
|z_u^i-x_v^i|
\le
|z_u^i-x_u^i|+|x_u^i-x_v^i|
\le
(1-\alpha)r_{k+1}+r_k
=
(1+\lambda-\alpha\lambda)r_k
\le
(1+\lambda)r_k;
\]
hence, we obtain
\begin{equation}\label{eq:radius-control}
\{z_u^i:u\in \children(v)\}\subset B\bigl(x_v^i,(1+\lambda)r_k\bigr).
\end{equation}

\medskip
\noindent\textbf{Step 3: the approximating sums are monotone.}
For each \(i\) and \(k\), set
\[
F_{i,k}:=\bigcup_{v\in \Tau_k^i} B(z_v^i,q_k)
\]
and define
\[
H_k:=F_{1,k}+\cdots+F_{n,k}.
\]

We claim that
\begin{equation}\label{eq:monotonicity}
H_k\subset H_{k+1}
\qquad (k\ge 0).
\end{equation}

To prove \eqref{eq:monotonicity}, we fix \(k\ge 0\) and a tuple of parent vertices
\[
(v_1,\dots,v_n)\in \Tau_k^1\times\cdots\times\Tau_k^n.
\]
For each \(i\), define the finite set
\[
A_i:=\{z_u^i:u\in \children(v_i)\}.
\]
By~\eqref{eq:q-ball},
\[
B(z_{v_i}^i,q_k)\subset \conv(A_i),
\]
and by~\eqref{eq:radius-control},
\[
A_i\subset B\bigl(x_{v_i}^i,R_k\bigr),
\qquad
R_k:=(1+\lambda)r_k.
\]
Moreover, the numerical hypothesis~\eqref{eq:parametric-bound} gives
\[
nq_{k+1}
=
n(\alpha-\lambda)r_{k+1}
=
n\lambda(\alpha-\lambda)r_k
>
\sqrt d\,(1+\lambda)r_k
=
R_k\sqrt d.
\]
Applying Lemma~\ref{lem:one-step-absorption} with \(b_i=z_{v_i}^i\), \(c_i=x_{v_i}^i\), \(q=q_k\), \(Q=q_{k+1}\), and
\(R=R_k\), we obtain
\begin{equation}\label{eq:parent-to-children}
B(z_{v_1}^1,q_k)+\cdots+B(z_{v_n}^n,q_k)
\subset
\left(\bigcup_{u_1\in \children(v_1)} B(z_{u_1}^1,q_{k+1})\right)+\cdots+
\left(\bigcup_{u_n\in \children(v_n)} B(z_{u_n}^n,q_{k+1})\right).
\end{equation}

Now, we take the union of~\eqref{eq:parent-to-children} over all parent tuples
\((v_1,\dots,v_n)\in \Tau_k^1\times\cdots\times\Tau_k^n\). Using distributivity of
Minkowski addition over unions, and the fact that
\[
\Tau_{k+1}^i=\bigsqcup_{v\in \Tau_k^i}\children(v)
\qquad (1\le i\le n),
\]
we find
\begin{align*}
H_k
&=
\bigcup_{(v_1,\dots,v_n)\in \Tau_k^1\times\cdots\times\Tau_k^n}
\Bigl(B(z_{v_1}^1,q_k)+\cdots+B(z_{v_n}^n,q_k)\Bigr)\\
&\subset
\bigcup_{(v_1,\dots,v_n)\in \Tau_k^1\times\cdots\times\Tau_k^n}
\left(
\bigcup_{u_1\in \children(v_1)} B(z_{u_1}^1,q_{k+1})
+\cdots+
\bigcup_{u_n\in \children(v_n)} B(z_{u_n}^n,q_{k+1})
\right)\\
&=
\left(\bigcup_{w_1\in \Tau_{k+1}^1} B(z_{w_1}^1,q_{k+1})\right)+\cdots+
\left(\bigcup_{w_n\in \Tau_{k+1}^n} B(z_{w_n}^n,q_{k+1})\right)\\
&=H_{k+1};
\end{align*}
thus proving~\eqref{eq:monotonicity}.

\medskip
\noindent\textbf{Step 4: passage to the limit.}
For every \(i\), \(k\), and \(v\in \Tau_k^i\), we have
\[
B(z_v^i,\alpha r_k)\subset \conv\bigl(\{x_u^i:u\in \children(v)\}\bigr),\qquad
\{x_u^i:u\in \children(v)\}\subset B(x_v^i,r_k),
\]
and hence
\[
B(z_v^i,\alpha r_k)\subset B(x_v^i,r_k).
\]
As in Step~1, the inclusion \(B(z_v^i,\alpha r_k)\subset B(x_v^i,r_k)\) implies
\[
|z_v^i-x_v^i|+\alpha r_k\le r_k,
\qquad\text{so}\qquad
|z_v^i-x_v^i|\le (1-\alpha)r_k.
\]
Thus, if \(y\in B(z_v^i,q_k)\), then by the triangle inequality
\[
|y-x_v^i|
\le
|y-z_v^i|+|z_v^i-x_v^i|
\le
q_k+(1-\alpha)r_k
=
(\alpha-\lambda)r_k+(1-\alpha)r_k
=
(1-\lambda)r_k.
\]
Since \(x_v^i\in E_i\), it follows that
\[
B(z_v^i,q_k)\subset E_i+B\bigl(0,(1-\lambda)r_k\bigr).
\]
Thus, we find
\[
F_{i,k}\subset E_i+B\bigl(0,(1-\lambda)r_k\bigr),
\]
and therefore, using \(B(0,a)+B(0,b)=B(0,a+b)\),
\[
H_k\subset
(E_1+\cdots+E_n)+B\bigl(0,n(1-\lambda)r_k\bigr).
\]
Since \(H_0\subset H_k\) for all \(k\ge 0\) by~\eqref{eq:monotonicity}, every point
\(p\in H_0\) satisfies
\[
\dist(p,\,E_1+\cdots+E_n)\le n(1-\lambda)r_k
\]
for every \(k\). Letting \(k\to\infty\) gives \(\dist(p,\,E_1+\cdots+E_n)=0\).
Because \(E_1+\cdots+E_n\) is compact (hence closed), this implies \(p\in E_1+\cdots+E_n\).
Hence, we see that
\[
H_0\subset E_1+\cdots+E_n.
\]
But we have
\[
H_0
=
B(z_\varnothing^1,q_0)+\cdots+B(z_\varnothing^n,q_0)
=
B\!\left(z_\varnothing^1+\cdots+z_\varnothing^n,\ nq_0\right),
\]
and \(q_0=(\alpha-\lambda)r_0>0\). Thus \(H_0\subset E_1+\cdots+E_n\) is a ball with positive radius, so
\(E_1+\cdots+E_n\) must have non-empty interior.
\end{proof}

\begin{remark}[The role of Shapley--Folkman]
The core of the argument is in Step~3, where the Shapley--Folkman theorem enters. The
goal is to show that refining the tree from depth~\(k\) to depth~\(k+1\) does not
shrink the sumset approximation~\(H_k\). At depth~\(k\), we maintain a ball of
radius~\(q_k\) around each~\(z\)-point, and Step~1 tells us that the children's
\(z\)-points at depth~\(k+1\) have a convex hull containing that ball. If these
finite point clouds were themselves convex, then the Minkowski sum of the convex hulls
would immediately contain the sum of the balls, and we would be done---but of course
they are not convex, as they are finite sets of points. The radius form of
Shapley--Folkman bridges this gap: the sum of the \(n\) point clouds differs from the
sum of their convex hulls by at most \(R_k\sqrt d\), where \(R_k=(1+\lambda)r_k\) is a
common radius bound for the clouds. Meanwhile the \(n\) summands collectively contribute
a total ball radius of~\(nq_{k+1}\) at the finer scale. The numerical condition on~\(n\)
is chosen precisely so that \(nq_{k+1}>R_k\sqrt d\), ensuring that the Shapley--Folkman
error is absorbed and \(H_k\subset H_{k+1}\).
\end{remark}

\begin{remark}[Understanding the constant]\label{rem:constants}
The threshold
\[
\frac{\sqrt d\,(1+\lambda)}{\lambda(\alpha-\lambda)}
\]
is exactly the condition needed in Step~3 when Lemma~\ref{lem:one-step-absorption} is applied with
\(b_i=z_{v_i}^i\), \(c_i=x_{v_i}^i\), \(q=q_k\), \(Q=q_{k+1}\), and \(R=R_k\) to the child clouds
\[
A_i:=\{z_u^i:u\in \children(v_i)\}
\]
associated to a fixed parent tuple \((v_1,\dots,v_n)\). Indeed, Step~1 gives
\[
B(z_{v_i}^i,q_k)\subset \conv(A_i),
\qquad q_k=(\alpha-\lambda)r_k,
\]
while Step~2 gives the common radius bound
\[
A_i\subset B(x_{v_i}^i,R_k),
\qquad R_k=(1+\lambda)r_k.
\]
The inflation parameter on the right-hand side of Lemma~\ref{lem:one-step-absorption} is
\[
Q=q_{k+1}=(\alpha-\lambda)r_{k+1}=\lambda(\alpha-\lambda)r_k,
\]
so the lemma's hypothesis \(R_k\sqrt d\le nQ\) becomes
\[
(1+\lambda)r_k\sqrt d\le n\lambda(\alpha-\lambda)r_k,
\]
equivalently
\[
n\ge \frac{\sqrt d\,(1+\lambda)}{\lambda(\alpha-\lambda)}.
\]
Proposition~\ref{prop:parametric-thick-sum} assumes the strict version of this inequality, which is exactly what yields \(nq_{k+1}>R_k\sqrt d\) in Step~3.
\end{remark}

\section{Optimization and proof of the main theorem}

Proposition~\ref{prop:parametric-thick-sum} includes two free parameters: \(\alpha\), which measures how
much of the thickness constant is retained at each scale, and \(\lambda\), the ratio
between successive scales in the tree. The constraint is \(0<\lambda<\alpha<c\). Since
\[
\Phi(\alpha,\lambda)=\frac{\sqrt d\,(1+\lambda)}{\lambda(\alpha-\lambda)}
\]
is strictly decreasing in~\(\alpha\), there is every incentive to take \(\alpha\) as
close to~\(c\) as possible. The main conceptual reason we do not set \(\alpha=c\)
outright is that the supremum in the definition of thickness need not be attained; in the
finite discretization step, we also need a small amount of slack to pass from a thick
compact piece to a finite point cloud. Once \(\alpha<c\) is fixed, the remaining task is
to optimize in \(\lambda\in(0,\alpha)\). Here \(\lambda\) governs both the next-scale
radius \(q_{k+1}=\lambda(\alpha-\lambda)r_k\) available for absorption and the radius
bound \(R_k=(1+\lambda)r_k\) appearing in the Shapley--Folkman error term. The
computation below identifies the optimal \(\lambda\), yielding Theorem~\ref{thm:main}.

\begin{proof}[Proof of Theorem~\ref{thm:main}]
Let
\[
\Phi(\alpha,\lambda):=\frac{\sqrt d\,(1+\lambda)}{\lambda(\alpha-\lambda)},
\qquad
0<\lambda<\alpha<c.
\]
By Proposition~\ref{prop:parametric-thick-sum}, it suffices to find \(\alpha,\lambda\) with
\(0<\lambda<\alpha<c\) and \(n>\Phi(\alpha,\lambda)\).

Fix \(\alpha\in (0,c)\). For this \(\alpha\), consider
\[
f_\alpha(\lambda):=\frac{1+\lambda}{\lambda(\alpha-\lambda)},
\qquad 0<\lambda<\alpha.
\]
A direct computation gives
\[
f_\alpha'(\lambda)=0
\quad\Longleftrightarrow\quad
\lambda^2+2\lambda-\alpha=0;
\]
the unique positive root is
\[
\lambda_\alpha=\sqrt{1+\alpha}-1.
\]
Moreover, we have
\[
\alpha-\lambda_\alpha=\lambda_\alpha(1+\lambda_\alpha)>0,
\]
so indeed \(0<\lambda_\alpha<\alpha\). Thus \(\lambda_\alpha\) is an interior critical
point of \(f_\alpha\) on \((0,\alpha)\). Since \(f_\alpha(\lambda)\to+\infty\) as
\(\lambda\downarrow 0\) and as \(\lambda\uparrow \alpha\), this critical point is the
unique global minimizer of \(f_\alpha\) on \((0,\alpha)\). At this critical point,
\[
\Phi(\alpha,\lambda_\alpha)
=
\frac{\sqrt d\,(1+\lambda_\alpha)}{\lambda_\alpha(\alpha-\lambda_\alpha)}
=
\frac{\sqrt d\,(1+\lambda_\alpha)}{\lambda_\alpha\cdot\lambda_\alpha(1+\lambda_\alpha)}
=
\frac{\sqrt d}{\lambda_\alpha^2}
=
\frac{\sqrt d}{(\sqrt{1+\alpha}-1)^2}.
\]
The function \(\alpha\mapsto \sqrt d/(\sqrt{1+\alpha}-1)^2\) is strictly decreasing on
\((0,\infty)\). Therefore, if
\[
n>\frac{\sqrt d}{(\sqrt{1+c}-1)^2},
\]
we may choose \(\alpha<c\) sufficiently close to \(c\) that
\(n>\Phi(\alpha,\lambda_\alpha)\). Proposition~\ref{prop:parametric-thick-sum} then applies, showing
that \(E_1+\cdots+E_n\) has non-empty interior.

Finally, rationalizing the denominator gives the equivalent form
\[
\frac{\sqrt d}{(\sqrt{1+c}-1)^2}
=
\frac{\sqrt d\,(\sqrt{1+c}+1)^2}{c^2},
\]
and since \(0<c\le 1\),
\[
\frac{\sqrt d\,(\sqrt{1+c}+1)^2}{c^2}
\le
\frac{\sqrt d\,(\sqrt2+1)^2}{c^2}
=
\frac{(3+2\sqrt2)\sqrt d}{c^2}
<
\frac{6\sqrt d}{c^2}. \qedhere
\]
\end{proof}

\section{Concluding remarks}

\subsection{The mechanism behind the improvement}
The improvement from Feng--Wu's cubic threshold to the present quadratic one, and the
further sharpening from a linear \(d\)-loss to a \(\sqrt{d}\)-loss within the
simultaneous-convexification strategy, can be understood as follows. In the Feng--Wu
proof, each generation of the tree construction requires a geometric step that replaces a
discrete point cloud by a ball, applied to one summand at a time. Each such replacement
incurs a multiplicative loss of order~\(c\) in the radius of the ball maintained by the
induction. Since the counting of summands required to absorb these losses ultimately
scales as~\(c^{-3}\), the cubic threshold emerges.

In our argument, the per-summand replacement step is avoided entirely:
Corollary~\ref{cor:almost-convex} simultaneously convexifies all \(n\) summands at a cost
of \(O(\sqrt d\,r_k)\), independent of~\(n\). More precisely, the relevant comparison at
generation~\(k\) is
\[
nq_{k+1}\sim n(\alpha-\lambda)\lambda\,r_k\sim nc^2r_k
\]
versus
\[
R_k\sqrt d\sim \sqrt d\,r_k,
\]
leading to the condition \(nc^2\gtrsim \sqrt d\) and hence
\(n\gtrsim \sqrt d\,c^{-2}\).

\subsection{Sharpness and dimension-dependence}
It is natural to ask whether the quadratic dependence \(n\gtrsim c^{-2}\) is optimal, and
whether the explicit factor of~\(\sqrt d\) can be improved or removed. While the present
paper does not completely settle either question, the proof does show that the
dimensional loss enters at a very specific point: the one-step absorption in Step~3
passes through Corollary~\ref{cor:almost-convex}, and hence through the radius form of the
Shapley--Folkman theorem. In this sense, the square-root dependence on~\(d\) is a
feature of the simultaneous-convexification route used here.

At the same time, as noted in Remark~\ref{rem:almost-convex-coarse}, the estimate
\(R\sqrt d\) is still a deliberately uniform bound. Moreover, the sets \(A_i\) arising
from the tree construction are highly structured, and that extra structure may permit
better estimates than the common-radius bound used here.

Accordingly, it would be interesting to know whether one can retain the quadratic
dependence on \(c^{-1}\) while sharpening the dimensional constant further, or even
removing the dimensional loss altogether by a different simultaneous-convexification
mechanism.

\subsection*{Acknowledgments}

Part of this work was conducted while I was visiting the Technological Innovation, Entrepreneurship, and Strategic Management (TIES) group at the MIT Sloan School of Management; I greatly appreciate their hospitality.

I used LLMs to assist with some of the computations and analysis in the preparation of this article, particularly GPT-5.4 Pro and Claude 4.6 Opus (both accessed via Poe with the support of Quora, where I am an advisor). I particularly appreciate helpful comments from Refine.ink. The problem, analysis, and eventual written form are my own; and of course any errors remain my responsibility.

\end{document}